\documentclass[12pt]{amsart}

\usepackage{graphicx}
\usepackage{fullpage}
\usepackage{etoolbox}
\usepackage[inline]{enumitem}
\usepackage[utf8]{inputenc}

\usepackage{amssymb}  
\usepackage{amsfonts}
\usepackage{amsthm}  
\usepackage{bbm}    
\usepackage{bm}
\usepackage{accents}
\usepackage{mathrsfs}
\usepackage[leqno]{mathtools}
\usepackage[leqno]{amsmath}
\usepackage{fixmath}
\usepackage{fullpage}
\usepackage{tikz}
\usepackage{microtype}
\usepackage[sharp]{easylist}

\usepackage{amssymb,latexsym} 
\usepackage{fullpage}
\usepackage{etoolbox}
\usepackage{enumitem}
\usepackage{color}
\usepackage[colorlinks,linkcolor=blue,citecolor=magenta]{hyperref}
\usepackage{soul}
\usepackage{microtype}
\usepackage{yhmath}
\usepackage{float}
\usepackage{algorithm}
\usepackage{bbold}
\usepackage{booktabs}
\usepackage{comment}
\usepackage[leqno]{mathtools}
\usepackage{calc}
\usepackage{complexity}

\newlength{\LPlhbox}

\usetikzlibrary{shapes}

\newtheorem*{theorem-nonumber}{Theorem}
\newtheorem{theorem}{Theorem}[section]
\newtheorem{proposition}[theorem]{Proposition}

\newtheorem{lemma}[theorem]{Lemma}

\theoremstyle{definition}
\newtheorem{example}{Example}

\theoremstyle{remark}
\newtheorem{remark}{Remark}

\makeatletter
\newcommand{\leqnomode}{\tagsleft@true}
\newcommand{\reqnomode}{\tagsleft@false}
\newcommand{\conv}{\operatorname{conv}}

\newcommand{\relint}{\operatorname{relint}}

\newcommand{\val}{\operatorname{val}}
\newcommand{\calC}{\mathcal{C}}
\newcommand{\B}{\mathcal{B}}
\newcommand{\bbR}{\mathbb{R}}
\newcommand{\x}{\boldsymbol{x}}
\newcommand{\y}{\boldsymbol{y}}

\newcommand{\boldt}{\boldsymbol{t}}
\newcommand{\uu}{\boldsymbol{u}}
\newcommand{\vv}{\boldsymbol{v}}
\newcommand{\boldb}{\boldsymbol{b}}
\newcommand{\boldc}{\boldsymbol{c}}
\newcommand{\boldd}{\boldsymbol{d}}

\newcommand{\zero}{\boldsymbol{0}}

\newcommand{\LPblocktag}[2]{\settowidth{\LPlhbox}{(#1)}%
	\parbox{\LPlhbox}{\begin{equation}\tag{#1}#2\end{equation}}%
	\hspace*{\fill}}

\makeatother
\renewcommand{\geq}{\geqslant}
\renewcommand{\leq}{\leqslant}

\title{Finite adaptability in two-stage robust optimization:
asymptotic optimality and tractability}

\author{Safia Kedad-Sidhoum \and Anton Medvedev \and Frédéric Meunier}

\address[Safia Kedad-Sidhoum]{CEDRIC, CNAM, France.}
\email{safia.kedad\_sidhoum@cnam.fr}

\address[Anton Medvedev]{CEDRIC,  CNAM \& SGR, France.}
\email{anton.medvedev@lecnam.net}

\address[Frédéric Meunier]{CERMICS, ENPC, Institut Polytechnique de Paris, Marne-la-Vallée, France.}
\email{frederic.meunier@enpc.fr}

\begin{document}

\begin{abstract}

Two-stage robust optimization is a fundamental paradigm for modeling and solving optimization problems with uncertain parameters. A now classical method within this paradigm is {\em finite adaptability}, introduced by Bertsimas and Caramanis (\emph{IEEE Transactions on Automatic Control}, 2010). It consists in restricting the recourse to a finite number $k$ of possible values. In this work, we point out that the continuity assumption they stated to ensure the convergence of the method when $k$ goes to infinity is not correct, and we propose an alternative assumption for which we prove the desired convergence. Bertsimas and Caramanis also established that finite adaptability is $\NP$-hard, even in the special case when $k=2$, the variables are continuous, and only specific parameters are subject to uncertainty. We provide a theorem showing that this special case becomes polynomial when the uncertainty set is a polytope with a bounded number of vertices, and we extend this theorem for $k=3$ as well. On our way, we establish new geometric results on coverings of polytopes with convex sets, which might be interesting for their own sake.
\end{abstract}




\maketitle

\section{Introduction}
Of particular importance in operations research are two-stage robust optimization problems of the form
 \begin{equation}\label{prob:comp-adapt}
 \tag{CompAdapt$(\Omega)$}
\begin{array}{rll}
\displaystyle\min_{\x,\y(\cdot)} & \displaystyle \boldc^{\top}\x + \max_{\omega \in \Omega}\boldd^{\top}\y(\omega) \\
\text{s.t.} & A(\omega)\x+B(\omega)\y(\omega)\leq \boldb(\omega) & \forall \omega \in \Omega \, ,
\end{array}
\end{equation}
where $\Omega$ is the uncertainty set, $A(\cdot)$, $B(\cdot)$, and $\boldb(\cdot)$ are input matrices and vector depending on the uncertainty, and $\boldc$ and $\boldd$ are deterministic vectors. We assume throughout the paper that $\Omega$ is a polytope of $\bbR^n$ and that the dependence 
to the uncertainty is affine. Even if more general cases are covered in the literature, these assumptions are standard. The variables are $\x$ and $\y(\cdot)$: the variable $\x$ is the ``here-and-now'' variable, whose value has to be determined without knowing the exact $\omega \in \Omega$ that will be selected, contrary to $\y(\cdot)$---the ``wait-and-see'' variable---whose value can arbitrarily depend on $\omega$. They can moreover be constrained to be integer, depending on the context.

These problems are especially relevant in situations where actions are still possible after some uncertainty has been revealed, and have been extensively studied in the literature~\cite{yanikouglu2019survey}. A fundamental technique to solve this problem, \emph{finite adaptability}, has been introduced in 2010 by Bertsimas and Caramanis in their seminal paper \emph{Finite Adaptability in Multistage Linear Optimization}~\cite{bertsimas2010finite}. It consists in restricting the range of $\y(\cdot)$ to piecewise constant functions with at most $k$ distinct values:
\begin{equation}\label{prob:finite-adapt}
\tag{Adapt$_k(\Omega)$}
\begin{array}{rll}
\displaystyle\min_{\substack{\Omega_1,\dots,\Omega_k \\ \x,\y_1,\ldots,\y_k}} & \boldc^\top \x + {\displaystyle\max_{i \in [k]}\boldd^{\top}\y_i} &\\[3ex]
\text{s.t.} & A(\omega)\x+B(\omega)\y_i\leq \boldb(\omega) & \forall i \in [k] \;\; \forall \omega \in \Omega_i\, ,
\end{array}
\end{equation}
where the $\Omega_i$ are constrained to form a cover of $\Omega$. A natural question is how well finite adaptability approximates {\em complete adaptability}, which consists in solving~\eqref{prob:comp-adapt} in its full generality, i.e., without any restriction on the variable $\y(\cdot)$. This question is addressed by Bertsimas and Caramanis, especially for the case when $k$ goes to infinity. Proposition~1 of the aforementioned paper states that, under a continuity assumption,
\begin{equation}\label{eq:lim}
\lim_{k \rightarrow +\infty}{\val}{\eqref{prob:finite-adapt}} = {\val}{\eqref{prob:comp-adapt}} \, , 
\end{equation}
where $\val$(P) denotes the optimal value of a problem (P). We claim that Proposition 1 is not correct, namely that the continuity assumption does not necessarily ensure the convergence of ${\val}{\eqref{prob:finite-adapt}}$ to ${\val}{\eqref{prob:comp-adapt}}$ when $k$ goes to infinity. We provide two counter-examples, with distinct behaviors, and propose an alternative continuity assumption, for which we prove the desired asymptotic result. Complementary results on how bad finite adaptability can behave have been proposed by Hanasusanto, Kuhn, and Wiesemann~\cite{hanasusanto2015k} and by El Housni and Goyal~\cite{el2018piecewise}.

In addition to this discussion about the asymptotic behavior of finite adaptability, our contributions concern its tractability aspects. From a computational perspective, only $\x,\y_1,\ldots,\y_k$ are expected to be output, and not the $\Omega_i$'s. These latter can be recovered from $\x$ and the $\y_i$'s and they are not needed to implement in practice a solution provided by the finite adaptability approach. In their paper, Bertsimas and Caramanis prove that \eqref{prob:finite-adapt} is \NP-hard, even when simultaneously $k=2$, the variables are continuous, and $A(\cdot)$ and $B(\cdot)$ are deterministic ($A(\omega)$ and $B(\omega)$ do not depend on $\omega$). Further hardness results have been proposed by Hanasusanto, Kuhn, and Wiesemann~\cite{hanasusanto2015k}. Regarding tractability, Subramanyam, Gounaris, and Wiesemann~\cite{subramanyam2020k} proved that \eqref{prob:finite-adapt} is polynomial when only the objective function is uncertain. They also propose an efficient branch-and-bound approach. Recently, Cardinal, Goaoc, and Wajsbrot~\cite{cardinal2025adapt} proved other tractability results, in particular that \eqref{prob:finite-adapt} is polynomially solvable when $k$ and the dimensions of the solution and uncertainty spaces are fixed.. As far as we know, there are no other theoretical works on the computational tractability of finite adaptability.

We prove the following two theorems, where $V(\Omega)$ is the set of vertices of $\Omega$ and $E(\Omega)$ is its set of edges. 
They show in particular that when $k\leq 3$, the number of vertices of $\Omega$ is bounded, $A(\cdot)$ and $B(\cdot)$ are deterministic, and there is no integrality constraint, \eqref{prob:finite-adapt} can be solved in polynomial time. Their proof relies on new results in discrete geometry dealing with the covering of a polytope with convex sets. Note that the case $k=1$ is immediate.

\begin{theorem}
\label{thm:k_2}
    Suppose that $k=2$. Then \eqref{prob:finite-adapt} with $A(\cdot)$ and $B(\cdot)$ deterministic can be solved by at most $3^{|V(\Omega)|}$ resolutions of a linear program of polynomial size (or a mixed linear program in case of integral variables).
\end{theorem}

This linear program and the way to use it for solving the problem is described in Section~\ref{sec:alg_2_3}. From this, it is actually rather straightforward to show that the problem \eqref{prob:finite-adapt}, under the condition of the theorem, can be solved by a single mixed integer linear program of polynomial size, with $O(|E(\Omega)|)$ binary variables. This shows that our result has also concrete consequences for practical purposes. See Section~\ref{sec:explicit_milp} for the detailed mixed integer linear program.

\begin{theorem}
\label{thm:k_3}
    Suppose that $k=3$. Then \eqref{prob:finite-adapt} with $A(\cdot)$ and $B(\cdot)$ deterministic can be solved by at most $7^{|V(\Omega)|+|E(\Omega)|}$ resolutions of a linear program of polynomial size (or a mixed linear program in case of integral variables).
\end{theorem}

Similarly as for the case $k=2$, it is not too difficult to show that under the condition of the theorem, \eqref{prob:finite-adapt} can be solved by a single mixed integer linear program of polynomial size, with a number of binary variables which is linear in the number of two-dimensional faces of $\Omega$. 

We conjecture that these theorems can be extended to any $k$ but we failed to find the proper generalization of our results in discrete geometry.

We also establish the following result, which implies a polynomial algorithm in the case of continuous variables.

\begin{proposition}
\label{prop:omega_dim_1}
    Suppose that $\Omega$ is one-dimensional. Then \eqref{prob:finite-adapt} with $A(\cdot)$ and $B(\cdot)$ deterministic can be solved by the resolution of a single linear program of polynomial size (or a single mixed linear program in case of integral variables).
\end{proposition}

The problems considered in Theorem~\ref{thm:k_2}, Theorem~\ref{thm:k_3}, Proposition~\ref{prop:omega_dim_1}, and the \NP-hardness result of Bertsimas and Caramanis mentioned above share a common feature: the matrices $A(\cdot)$ and $B(\cdot)$ are deterministic. We remark that the complete adaptable counterpart of these problems can actually be solved in polynomial time with respect to the number of vertices of $\Omega$. Indeed, by convexity it boils down to only determine $\y(\cdot)$ for the vertices of $\Omega$. In other words, for such problems,~\eqref{prob:comp-adapt} is easier to solve than~\eqref{prob:finite-adapt}. It is the study of finite adaptability per se that motivates our results and the \NP-hardness result of Bertsimas and Caramanis. Moreover, it seems reasonable to imagine problems where the decision maker wants to reduce the number of possible values of ``wait-and-see'' variables, making the finite adaptable setting the most suitable one.

\section{Asymptotic optimality}

The continuity assumption of the paper by Bertsimas and Caramanis~\cite[Section 3]{bertsimas2010finite} is
\begin{quote}
    \emph{{\bf \emph{Continuity assumption:} } For any $\varepsilon>0$, for any $\omega \in \Omega$, there exist $\delta>0$ and a point $(\x,\y)$, feasible for $A(\omega), B(\omega)$ and within $\varepsilon$ of optimality, such that $(\x,\y)$ is also feasible for $A(\omega'),B(\omega')$ for all $\omega' \in \Omega$ with $\|\omega-\omega'\|\leq \delta$.}
\end{quote}
A point $(\x,\y)$ is {\em within $\varepsilon$ of optimality} if the quantity $\boldc^{\top}\x + \boldd^{\top}\y - \varepsilon$ is at most ${\val}{\eqref{prob:comp-adapt}}$. We provide two counter-examples to the statement that this continuity assumption ensures the convergence given in equation~\eqref{eq:lim}.

\begin{example}[Counter-example with a finite gap]
\label{ex:example1}
Consider the following two-stage robust optimization problem:

\reqnomode
\LPblocktag{P}{\label{prob:ex_real}}%
\hspace{-12cm}
\begin{minipage}{\linewidth-1cm}
	\begin{align}\notag
\min\quad & x_1 \\
\text{s.t.}\quad & x_2 - x_3 \leq x_1 \label{eq:p_1} \\
& x_3 \leq \omega x_4\leq x_2 & \forall \omega \in [0,1] \label{eq:p_2} \\
&(\omega + 2)y(\omega) + (3 - \omega)x_4 = 10 & \forall \omega \in [0,1] \label{eq:p_3}\\
\notag &x_1,x_2,x_3,x_4\geq 0, \; y\colon[0,1]\rightarrow\mathbb{R}_+\, .
\end{align}
\end{minipage}

\vspace{0.33cm}

\noindent It is a special case of~\eqref{prob:comp-adapt}. An optimal solution is given by $x_1 = x_2 = x_3 = x_4 = 0$ and $y(\omega)=\frac{10}{\omega + 2}$, and thus ${\val}{\eqref{prob:ex_real}}=0$.

Problem~\eqref{prob:ex_real} satisfies the continuity assumption, as we check now. Let $\varepsilon>0$ and $\omega \in [0,1]$. Set
\[\delta \coloneqq \frac{\varepsilon}{4}, \quad x_1 \coloneqq \varepsilon, \quad x_2 \coloneqq 2(\omega + \delta), \quad x_3 \coloneqq 2(\omega - \delta), \quad x_4 \coloneqq 2, \quad y \coloneqq 2 \, .
\]
The point $(x_1,x_2,x_3,x_4,y)$ satisfies the constraints of~\eqref{prob:ex_real} for all $\omega'\in[0,1]$ such that $|\omega-\omega'|\leq \delta$ and is within $\varepsilon$ of optimality. Therefore, the continuity assumption is satisfied.

Consider now a feasible solution of the finite adaptability version of Problem~\eqref{prob:ex_real} for some $k$. One $\Omega_i$ of this solution must contain at least two distinct points of the uncertainty set $[0,1]$. By consequence, for this $\Omega_i$ the only value of $(x_4,y)$ satisfying constraint~\eqref{eq:p_3} is $x_4=y=2$. As $x_4=2$, constraints~\eqref{eq:p_1}~and~\eqref{eq:p_2} imply that the objective value is lower bounded by $2$. This latter value is actually achieved by $x_1 = x_2 = x_4 = y = 2$ and $x_3=0$.

So, even though the continuity assumption holds, the equality~\eqref{eq:lim} is not verified: the left-hand term is $2$, while the right-hand term is $0$.
\end{example}

In Example~\ref{ex:example1} the gap between the optimal value of Problem~\eqref{prob:ex_real} and the optimal value of its finite adaptability version is equal to $2$. Adding a constraint $x_1\leq 1$ in Problem~\eqref{prob:ex_real} makes this gap infinite. Below we propose another example with an infinite gap that does not come from a trivial constraint. 

\begin{example}[Counter-example with an infinite gap] Consider the following two-stage robust optimization problem:

\reqnomode
\LPblocktag{Q}{\label{prob:ex}}%
\hspace{-12cm}
\begin{minipage}{\linewidth-1cm}
	\begin{align}\notag
\min\quad & x_1 \\
\text{s.t.}\quad &\omega - x_1 \leq y(\omega) \leq \omega+x_1 & \forall \omega \in [0,1]\label{eq:q_1} \\
&x_1(\omega - 0.1) \leq x_2 \leq x_1(\omega + 0.1) & \forall \omega \in [0,1] \label{eq:q_2} \\
\notag & x_1,x_2\in[0,1], \; y\colon[0,1]\rightarrow[0,1]\, .
\end{align}
\end{minipage}

\vspace{0.33cm}

\noindent It is a special case of~\eqref{prob:comp-adapt}. An optimal solution is given by $x_1 = x_2 = 0$ and $y(\omega)=\omega$, and thus ${\val}{\eqref{prob:ex}}=0$.

Problem~\eqref{prob:ex} satisfies the continuity assumption, as we check now. Let $\varepsilon>0$ and $\omega \in [0,1]$. Set
\[\delta \coloneqq \min(0.1,\varepsilon), \quad x_1 \coloneqq \varepsilon, \quad x_2 \coloneqq \varepsilon\omega, \quad y \coloneqq \omega \, .
\]
The point $(x_1,x_2,y)$ satisfies the constraints of~\eqref{prob:ex} for all $\omega'\in[0,1]$ such that $|\omega-\omega'|\leq \delta$ and is within $\varepsilon$ of optimality. Therefore, the continuity assumption is satisfied.

There is no value of $k$ for which the finite adaptability version of Problem~\eqref{prob:ex} is feasible, as we prove now. If $x_1=0$, then the constraint~\eqref{eq:q_1} imposes that the (finitely many) $y_i$'s take all values in $[0,1]$, which is impossible. If $x_1 > 0$, then the constraint~\eqref{eq:q_2} imposes that simultaneously $ x_2 \leq 0.1x_1$ (obtained with $\omega=0$) and $0.9x_1 \leq x_2$ (obtained with $\omega=1$), which is impossible.

So, even though the continuity assumption holds, the equality~\eqref{eq:lim} is not verified: the left-hand term is $+\infty$, while the right-hand term is $0$. Here, although the continuity assumption is verified and the complete adaptability version is feasible, the finite adaptability version is never feasible.
\end{example}
We propose instead the following continuity assumption.
\begin{quote}
    \emph{{\bf \emph{Modified continuity assumption:} } For any $\varepsilon>0$, there exists $\x$ such that, for every $\omega \in \Omega$, there exist $\delta>0$ and $\y$ satisfying the following two conditions simultaneously:
    \begin{itemize}[leftmargin=*]
        \item[$\bullet$] $(\x,\y)$ is feasible for $A(\omega'),B(\omega')$ for all $\omega' \in \Omega$ with $\|\omega-\omega'\|\leq \delta$.
        \item[$\bullet$] $(\x,\y)$ is within $\varepsilon$ of optimality.
    \end{itemize}
        }
\end{quote}
This ``modified continuity assumption'' is more restrictive in the sense that every problem satisfying the ``modified continuity assumption'' satisfies the original ``continuity assumption.''
Problem~\eqref{prob:ex} can be used to see that the converse is not true. Note however that if there is no ``here-and-now'' variable $\x$, i.e., $\boldc=\zero$ and $A(\omega)$ is the zero-matrix for all $\omega$, then the two assumptions coincide.

We establish now a version of the Proposition~1 of the paper by Bertsimas and Caramanis~\cite{bertsimas2010finite} relying on the modified assumption. This shows that the general message of their Proposition~1 is correct, namely that, under some continuity assumption, finite adaptability converges to complete adaptability.

\begin{theorem}
Assume that $\Omega$ is a compact subset of $\bbR^n$. If the ``modified continuity assumption'' holds, then 
\[
\lim_{k \rightarrow +\infty}{\val}{\eqref{prob:finite-adapt}} = {\val}{\eqref{prob:comp-adapt}} \, .
\]
\end{theorem}

\begin{proof}
Let $\varepsilon>0$. The ``modified continuity assumption'' ensures then the existence of an $\x$ and, for every $\omega\in\Omega$, of a $\delta^{\omega}$ and a $\y^{\omega}$ such that:
\begin{enumerate}[label=(\roman*)]
    \item\label{feas} $A(\omega')\x + B(\omega')\y^{\omega} \leq \boldb(\omega')$ for every $\omega' \in \B(\omega,\delta^{\omega})$.
    \item\label{opt} $\boldc^{\top}\x + \boldd^{\top}\y^{\omega} \leq {\val}{\eqref{prob:comp-adapt}} + \varepsilon $.
\end{enumerate}
(Here, $\B(\omega,\delta)$ stands for the open ball of radius $\delta$ centered at $\omega$.) The balls $\B(\omega,\delta^{\omega})$, for $\omega\in\Omega$, form a cover of $\Omega$. By compactness of this latter set, there is a collection of finitely many such balls $\B(\omega,\delta^{\omega})$ forming already a cover of $\Omega$. In other words, there exist $\omega_1,\ldots,\omega_k$ in $\Omega$, for some positive integer $k$, such that $\bigcup_{i=1}^k \B(\omega_i,\delta^{\omega_i}) = \Omega$. Set $\y_i\coloneqq \y^{\omega_i}$ and $\Omega_i \coloneqq \B(\omega_i,\delta^{\omega_i})$. By \ref{feas}, we have $A(\omega)\x+B(\omega)\y_i\leq \boldb(\omega)$ for every $i \in [k]$ and every $\omega \in \Omega_i$. This means that $\Omega_1,\ldots,\Omega_k$ and $\x,\y_1,\ldots,\y_k$ form a feasible solution of \eqref{prob:finite-adapt}. By \ref{opt}, we have then ${\val}{\eqref{prob:finite-adapt}}\leq {\val}{\eqref{prob:comp-adapt}} + \varepsilon$.
This means that for every $\varepsilon > 0$, there exists an integer $k$ such that ${\val}{\eqref{prob:finite-adapt}}\leq {\val}{\eqref{prob:comp-adapt}} + \varepsilon$. On the other hand, by definition of \eqref{prob:finite-adapt}), the quantity ${\val}{\eqref{prob:finite-adapt}}$ is non-increasing in $k$, and is lower bounded by \eqref{prob:comp-adapt}. Thus, the limit $\lim_{k \rightarrow +\infty}{\val}{\eqref{prob:finite-adapt}}$ exists, and is lower bounded by ${\val}{\eqref{prob:comp-adapt}}$ and upper bounded by ${\val}{\eqref{prob:comp-adapt}} + \varepsilon$ for all $\varepsilon>0$. The desired conclusion follows then immediately.
\end{proof}

\begin{remark}
    The proof makes clear that the convergence result remains true under ``modified continuity assumption'' for arbitrary dependence of $A(\cdot)$, $B(\cdot)$, and $\boldb(\cdot)$ to the uncertainty.
\end{remark}

\section{Tractable cases}

Assuming convexity of the $\Omega_i$'s in the cover of \eqref{prob:finite-adapt} does not increase the optimal value; see Section~\ref{subsubsec:cover} below. In the setting of Theorems~\ref{thm:k_2} and~\ref{thm:k_3}, much stronger assumptions can actually be done on this cover. In essence, the $\Omega_i$ can then be assumed to be polytopes with not too many vertices and that brings us to finitely many configurations. Proving that these further  properties can be safely assumed is the objective of Section~\ref{sec:geom}. In Section~\ref{sec:alg}, this is turned into an algorithm that consists in enumerating all possible configurations. In that section, we also prove Proposition~\ref{prop:omega_dim_1}.

\subsection{Geometric preliminaries}
\label{sec:geom}

\subsubsection{Cover, partition, and convexity}\label{subsubsec:cover}

The next lemma is probably common knowledge in the area of robust optimization but we are not aware of any reference in the literature. For sake of completeness, we provide a formal statement together with a proof, and an illustration.

\begin{lemma}\label{lem:convexity}
The $\Omega_i$'s in \eqref{prob:finite-adapt} can be assumed to be closed and convex.
\end{lemma}

In other words, for every feasible solution of \eqref{prob:finite-adapt}, there exists a feasible solution with the same $\x$ and $\y_i$'s, and with closed and convex $\Omega_i$'s. Note that these two solutions have necessarily the same objective value. We believe that we can further assume that the $\Omega_i$'s are actually polytopes. (This is actually established below implicitly for $k \leq 3$. For $k=3$, this results from  Lemmas~\ref{lem:cover},~\ref{lem:cover-2}, and~\ref{lem:3conv}.) Even if we failed to find a simple proof of this fact for all $k$, we think that this should not be too difficult. Since this is not required to establish our main results, we did not push further in that direction.

\begin{proof}[Proof of Lemma~\ref{lem:convexity}]
Let $\Omega_1,\ldots,\Omega_k,\x,\y_1,\ldots,\y_k$ be a feasible solution of the problem \eqref{prob:finite-adapt}. Define  $\Omega_i' \coloneqq \{ \omega \in \Omega \colon A(\omega)\x+B(\omega)\y_i\leq \boldb(\omega)\}$. Since $A$, $B$, and $\boldb$ are affine, the sets $\Omega_i'$ are closed and convex. They form a cover of $\Omega$ because the $\Omega_i$ form such a cover and $\Omega_i \subseteq \Omega_i'$ for every $i$. Together with $\x$ and $(\y_1,\ldots,\y_k)$, they form a feasible solution of \eqref{prob:finite-adapt}, as desired. 
\end{proof}

In the literature, the problem~\eqref{prob:finite-adapt} is sometimes formulated with the $\Omega_i$ only constrained to form a partition of $\Omega$. The two formulations are completely equivalent, as we briefly explain now. A cover being more general than a partition, the value of the problem with the ``cover'' constraint is at most the value with the ``partition'' constraint. The value of the problem with the ``cover'' constraint is also at least the value with the ``partition'' constraint because allocating each $\omega$ in two or more $\Omega_i$'s to one of them arbitrarily keeps feasibility. However, the formulation with the ``cover'' constraint allows an extra convexity requirement, as shown by Lemma~\ref{lem:convexity}, while this is not necessarily the case with the ``partition'' constraint.

\begin{example}
Consider the following two-stage robust optimization problem:

\reqnomode
\LPblocktag{R}{\label{prob:cover_no_partition}}%
\hspace{-12cm}
\begin{minipage}{\linewidth-1cm}
	\begin{align}\notag
\min_{\substack{\Omega_1,\Omega_2 \\ y_1,y_2}}\quad & 0 \\
\text{s.t.}\quad &\notag(\omega_2 - \omega_1)y_i \leq 1-\omega_1 & \forall i \in \{1,2\} \;\; \forall (\omega_1,\omega_2) \in \Omega_i \\
&\notag(\omega_1 - \omega_2)y_i \leq 1+\omega_1 & \forall i \in \{1,2\} \;\; \forall (\omega_1,\omega_2) \in \Omega_i  \\
\notag & y_1,y_2 \in \{0, 1\}\, ,
\end{align}
\end{minipage}
\vspace{0.33cm}

\noindent with $\Omega = \conv\{(-2,0),(0,-2),(2,0),(0,2)\}$. Set 
\begin{align*}
\Omega_1 & \coloneqq \conv\{(-1,1),(0,2),(1,1),(1,-1),(0,-2),(-1,-1)\}, \, \\
\Omega_2 & \coloneqq \conv\{(-1,1),(1,1),(2,0),(1,-1),(-1,-1),(-2,0)\}\, .
\end{align*}
Figure~\ref{fig:omegas} represents $\Omega_1$ and $\Omega_2$.
Up to reindexing, these $\Omega_1$ and $\Omega_2$ form with  $(y_1,y_2) = (0,1)$ the only feasible solution of the problem~\eqref{prob:cover_no_partition} with convex $\Omega_i$'s. Thus, if the problem~\eqref{prob:cover_no_partition} is considered with the ``partition'' constraint, there is no solution with convex $\Omega_i$'s.

\begin{figure}[htbp]
    \centering
        \includegraphics[width=0.6\textwidth]{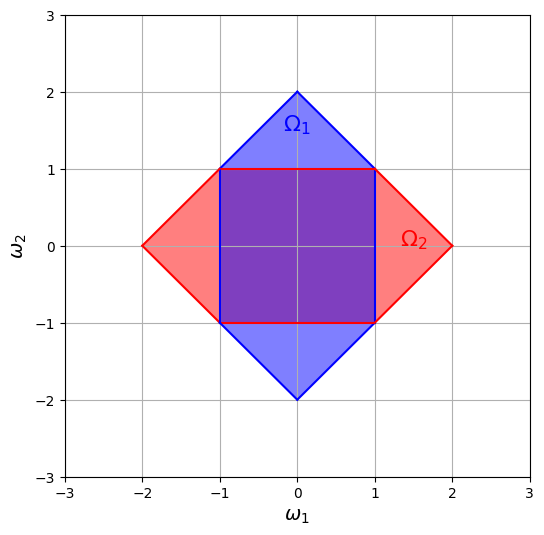}
    \caption{Representation of an optimal cover of problem~\eqref{prob:cover_no_partition}}
    \label{fig:omegas}
\end{figure}
\end{example}

\subsubsection{Covering polytopes with closed sets}

In the next section, we will establish geometric results that will allow assuming further that the $\Omega_i$'s in \eqref{prob:finite-adapt} are polytopes with extra properties, when $k=3$. We will rely on two geometric lemmas we state and prove now. They are actually direct consequences (with close proofs) of a classical 
result from combinatorial topology due to Fan.

\begin{theorem-nonumber}[{Fan's theorem~\cite[Theorem 2]{fan1952generalization}}]
    \textit{If $m$ closed sets $B_1,B_2,\ldots,B_m$ cover a $d$-dimensional sphere and if none of them contain antipodal points, then there exist $d+2$ indices $1 \leq i_1 < i_2 < \cdots < i_{d+2} \leq m$ such that
    \[
    B_{i_1} \cap -B_{i_2} \cap \cdots \cap (-1)^{d+1}B_{i_{d+2}} \neq \varnothing \, ,
    \] where $-B_i$ denotes the antipodal set of $B_i$.}
\end{theorem-nonumber}

Note that this theorem implies in particular the following: {\em if $d+1$ closed sets cover the $d$-dimensional sphere, then at least one of them contains antipodal points.} This latter fact is nothing else than the celebrated Lyusternik--Shnirel'mann theorem~\cite{liusternik1934methodes}, one of the many versions of the Borsuk--Ulam theorem.

\begin{lemma}
\label{lem:cover_polytope}
    Assume the boundary of a polytope be covered by $d$ closed sets. If the dimension of the polytope is at least $d$, then the convex hulls of these closed sets cover the polytope.
\end{lemma}

\begin{proof}
    Let $F_1,F_2,\ldots,F_d$ be closed sets covering the boundary of a polytope $\Omega'$, and suppose that the dimension of the latter is at least $d$. Let $\x$ be a point of the relative interior of $\Omega'$. Consider an arbitrary $d$-dimensional affine subspace $H$ containing $\x$ and contained in the affine span of $\Omega'$. Let $S^{d-1}$ be a small sphere contained in $H \cap \relint \Omega'$ and centered at $\x$. Define $B_i$ as the points $\y$ in $S^{d-1}$ such that the ray originating at $\x$ and going through $\y$ intersects $\partial \Omega'$ in $F_i$. The sets $B_1, B_2,\ldots, B_d$ are closed and form a cover of $S^{d-1}$. By Fan's theorem, or more precisely by its corollary discussed just after the statement, one of the $B_i$'s contains two antipodal points, which implies that $\x$ is in the segment joining two points of $\partial \Omega'$ belonging to a same $F_i$. Hence, $\x$ belongs to $\conv(F_i)$. Since $\x$ was chosen arbitrary, every point of $\Omega'$ belongs to some $\conv(F_i)$, as desired.
\end{proof}

\begin{lemma}\label{lem:colorful_cover}
    Assume the boundary of a polygon be covered by three closed sets. If two of these closed sets do not intersect, then the convex hulls of these closed sets cover the polygon.
\end{lemma}

\begin{proof}
        Let $F_1,F_2,F_3$ be closed sets covering the boundary of a polygon $\Omega'$. Suppose w.l.o.g. that $F_1 \cap F_3 = \varnothing$. Let $\x$ be a point of the relative interior of $\Omega'$. Let $S^1$ be a small circle contained in $\relint \Omega'$ and centered at $\x$. Define $B_i$ as the points $\y$ in $S^1$ such that the ray originating at $\x$ and going through $\y$ intersects $\partial \Omega'$ in $F_i$. The sets $B_1, B_2,B_3$ are closed and form a cover of $S^1$. Since $F_1 \cap F_3 = \varnothing$, we have $B_1 \cap B_3 = \varnothing$ as well. By Fan's theorem
    , one of the $B_i$'s contains two antipodal points, which implies that $\x$ is in the segment joining two points of $\partial \Omega'$ belonging to a same $F_i$. Hence, $\x$ belongs to $\conv(F_i)$. Since $\x$ was chosen arbitrary, every point of $\Omega'$ belongs to some $\conv(F_i)$, as desired.
\end{proof}



We finish this section with Berge's theorem, which will also be used in the next section. It replaces the dimension condition of the celebrated Helly theorem (see, e.g.,~\cite{danzer1963helly}) by a condition on the union.

\begin{theorem-nonumber}[Berge theorem {\cite[Theorem 1.1]{barany2006berge}}]
    \textit{Consider a family of $m$ convex sets in $\bbR^d$ whose union is convex. If every $m-1$ members have a common point, then the whole family has a common point.}
\end{theorem-nonumber}

\subsubsection{Covering polytopes with three convex sets}\label{subsubsec:three}

This section proves useful lemmas to establish the validity of our algorithms in the case $k=3$. Similar properties can also be assumed when $k=2$, but everything is then much simpler and this will be briefly discussed in Section~\ref{sec:alg}.

For each edge of $\Omega$, we choose an arbitrary orientation. So, from now on, each edge of $\Omega$ is directed and this makes the $1$-skeleton of $\Omega$ a directed graph. We introduce the following object that will play a crucial role in our proofs. Consider a cover of $V(\Omega)$ with three sets $V_1,V_2,V_3$ and a labeling $\lambda$ of the edges $e$ of $\Omega$ with non-empty subsets of $\{1,2,3\}$ satisfying the following conditions:
\begin{enumerate}[label=(\alph*)]
    \item\label{single} if $\lambda(e)=\{i\}$, then the two endpoints of $e$ belongs to $V_i$.
    \item\label{pair} if $\lambda(e)=\{i,j\}$ ($i\neq j$), then $V_i$ contains one endpoint of $e$ and $V_j$ contains the other endpoint.
    \item\label{three} if $\lambda(e)=\{1,2,3\}$, then one of the $V_i$'s contains one endpoint of $e$, another one contains the other endpoint, and the last one contains no endpoint.
\end{enumerate}
Consider moreover a point $t_e$ in each edge $e$ with $|\lambda(e)|=2$, and two points $u_e$ and $v_e$ in each edge $e$ with $|\lambda(e)| = 3$, the tail of $e$, the point $u_e$, the point $v_e$, and the head of $e$ being in this order on $e$ (with possibly equality between $u_e$ and $v_e$). We call the tuple $\calC \coloneqq \bigl((V_i),\lambda,(t_e),(u_e),(v_e)\bigl)$ a {\em nice $1$-skeleton cover} of $\Omega$. The name finds its motivation from Lemma~\ref{lem:cover} below.

For a nice $1$-skeleton cover $\calC =  \bigl((V_i),\lambda,(t_e),(u_e),(v_e)\bigl)$ of $\Omega$, we define
\[
    \begin{array}{rl}
    \overline{V_i}(\calC) \: \coloneqq & V_i \cup \bigl\{t_e \colon i \in \lambda(e) \text{ and } |\lambda(e)| = 2\bigl\} \\[1ex] &\cup\, \,\bigl\{u_e,v_e \colon  V_i \text{ contains no endpoint of $e$ and } |\lambda(e)| = 3 \} \\[1ex] &\cup\, 
    \bigl\{u_e \colon V_i \text{ contains the tail of $e$ and } |\lambda(e)| = 3 \bigl\} \\[1ex] & \cup\,   \bigl\{v_e \colon V_i \text{ contains the head of $e$ and } |\lambda(e)| = 3 \bigl\}  \, .
    \end{array}
    \]

\begin{lemma}\label{lem:cover}
   The sets $\conv(\overline{V_i}(\calC))$ form a cover of the $1$-skeleton of $\Omega$.
\end{lemma}

\begin{proof}
First we prove that the $\conv(\overline{V_i}(\calC))$ form a cover of the $1$-skeleton of $\Omega$.
Consider an edge $e$ of $\Omega$. If the two endpoints of $e$ belong to a same $V_i$, then we are done by convexity. Assume thus that the two endpoints of $e$ do not belong to a same $V_i$. By definition of $\lambda$, we have $|\lambda(e)| \geq 2$. Suppose first $|\lambda(e)| = 2$. We can write $\lambda(e) = \{t,h\}$ such that $V_t$ contains the tail of $e$ and $V_h$ contains the head of $e$. We have then $e = [V_t\cap e,t_e] \cup [t_e, V_h\cap e]$. Since $t_e \in \overline{V_t}(\calC) \cap \overline{V_h}(\calC)$ by construction, we are done by convexity. Suppose then $|\lambda(e)| = 3$. Then it is possible to choose $t,h \in \{1,2,3\}$ such that $V_t$ contains the tail of $e$ and $V_h$ contains the head of $e$. Denote by $\ell$ the index distinct from $t$ and $h$. (Notice that $t \neq h$ since we are in a case where no $V_i$ contains both endpoints of $e$.) We have $e = [V_t \cap e, u_e] \cup [u_e,v_e] \cup [v_e,V_h \cap e]$. Since $u_e \in \overline{V_t}(\calC) $, $v_e \in \overline{V_h}(\calC) $, and $u_e,v_e \in \overline{V_{\ell}}(\calC)$ by construction, we are done by convexity.
\end{proof}

Consider a cover of the $1$-skeleton of $\Omega$ with three non-empty closed convex sets $C_1,C_2,C_3$. We define $F(C_1,C_2,C_3)$ as the set of two-dimensional faces of $\Omega$ such that the $C_i$'s are pairwise intersecting on the boundary of the face. 

We define similarly the notation $F(V_1,V_2,V_3,\lambda)$ as follows. Consider a nice $1$-skeleton cover $\calC$ of $\Omega$, with the $t_e,u_e,v_e$ in the relative interiors of the edges $e$, and with $u_e \neq v_e$ for all $e$. Then set \[
F(V_1,V_2,V_3,\lambda)\coloneqq F\bigl(\conv(\overline{V_1}(\calC)),\conv(\overline{V_2}(\calC)),\conv(\overline{V_3}(\calC))\bigl) \, .
\]
This is well-defined because the intersection pattern of the $\conv(\overline{V_i}(\calC))$'s clearly does not depend on the exact location of the $t_e,u_e,v_e$. Note that the following relation holds more generally
\begin{equation}\label{eq:2f}
F\bigl(\conv(\overline{V_1}(\calC)),\conv(\overline{V_2}(\calC)),\conv(\overline{V_3}(\calC))\bigl) \subseteq F(V_1,V_2,V_3,\lambda) \, ,
\end{equation}
with equality when the $t_e$, $u_e$, and $v_e$ are in the relative interior of $e$, and $u_e$ and $v_e$ are distinct. (The inclusion may be strict since the sets $\conv(\overline{V_1}(\calC))$, $\conv(\overline{V_2}(\calC))$, $\conv(\overline{V_3}(\calC))$ may have a common point.)

\begin{lemma}\label{lem:cover-2}
    Consider a cover of the $1$-skeleton of $\Omega$ with three closed convex sets $C_1,C_2,C_3$. Assume moreover that we are given a point $t_f$ in each two-dimensional face $f$ in $F(C_1,C_2,C_3)$. Then the sets $\conv\bigl(C_i \cup \{t_f\colon f \in F(C_1,C_2,C_3)\}\bigl)$ form a cover of $\Omega$.
\end{lemma}

\begin{proof}
   The sets $\conv\bigl(C_i \cup \{t_f\colon f \in F(C_1,C_2,C_3)\}\bigl)$ cover  the $1$-skeleton of $\Omega$, as $C_1,C_2,C_3$ already cover the $1$-skeleton of $\Omega$. Consider now a face $f$ of $\Omega$ of dimension at least two (the face $f$ can be $\Omega$ itself). We prove that $f \subseteq \bigcup_{i \in \{1,2,3\}}\conv\bigl(C_i \cup \{t_f\colon f \in F(C_1,C_2,C_3)\}\bigl)$ by dealing with the case where the dimension of $f$ is two separately, and by dealing with the other cases by induction.
    
Suppose first $f$ is a two-dimensional face of $\Omega$. If $f$ belongs to $F(C_1,C_2,C_3)$, then the sets $\conv\bigl(C_i \cup \{t_f\colon f \in F(C_1,C_2,C_3)\}\bigl)$ cover $f$ since $t_f$ is included in each 
of them and since the $1$-skeleton of $f$ is already covered. If $f$ does not belong to $F(C_1,C_2,C_3)$, 
then $f$ is covered by the sets $C_i$ (by Lemma~\ref{lem:colorful_cover} applied on the restriction of the $C_i$'s on the boundary of $f$).

Suppose then that the dimension of $f$ is at least three. By induction, the boundary of $f$ is covered by the three closed sets $\conv\bigl(C_i \cup \{t_f\colon f \in F(C_1,C_2,C_3)\}\bigl)$. These three sets are closed, and the desired conclusion follows from Lemma~\ref{lem:cover_polytope}. (This lemma is stated with $d$ closed sets, but the empty set being closed, we can apply it here.)
\end{proof}

Berge's theorem is used in the proof of the next lemma.

\begin{lemma}\label{lem:3conv}
Let $\Omega_1$, $\Omega_2$, $\Omega_3$ be three closed convex sets covering $\Omega$. Then there exists a nice $1$-skeleton cover $\calC = \bigl((V_i),\lambda,(t_e),(u_e),(v_e)\bigl)$ such that $\overline{V_i}(\calC) \subseteq \Omega_i$ for all $i$ and such that the $\Omega_i$'s have a common point on every face $f$ in $F(V_1,V_2,V_3,\lambda)$.
\end{lemma}

\begin{proof}
The result being obvious when the dimension of $\Omega$ is $0$, suppose that the dimension of $\Omega$ is at least $1$.

Set $V_i \coloneqq V(\Omega) \cap \Omega_i$ for $i=1,2,3$. For each edge $e$, define $\lambda$ as follows. If both endpoints belong to a same $V_i$, pick arbitrarily such a set and define $\lambda(e)$ to be its index. If both endpoints do not belong to a same $V_i$ but there is a pair of $\Omega_i$'s covering the edge, pick arbitrarily such a pair (there might indeed be some choice), and define $\lambda(e)$ to be the corresponding indices. Since the $\Omega_i$ are closed, the $\Omega_i$ in the pair have a common point. We pick such a point and call it $t_e$. Finally, if both endpoints do not belong to a same $V_i$ but the three $\Omega_i$'s are needed to cover the edge, define $\lambda(e)$ to be $\{1,2,3\}$. Again because of the closedness of the $\Omega_i$'s, it is possible to pick one point $u_e$ in $\Omega_t \cap \Omega_{\ell}$ and one point $v_e$ in $\Omega_h \cap \Omega_{\ell}$, where $\Omega_t$ contains the tail of $e$, $\Omega_h$ contains the head of $e$, and $\ell \in \{1,2,3\}\setminus\{t,h\}$. (Such indices are defined unambiguously.)

It is straightforward to check that conditions~\ref{single} and~\ref{pair} hold. Condition~\ref{three} holds as well. Indeed, suppose $|\lambda(e)|=3$.  Each endpoint of $e$ belongs to a distinct $V_i$. Moreover, it is easy to see that if every endpoint were in a $V_i$, then the convexity of the $\Omega_i$'s would imply that two of the $\Omega_i$'s would already cover $e$, which would prevent setting $\lambda(e)$ to be $\{1,2,3\}$. The tuple $\calC = \bigl((V_i),\lambda,(t_e),(u_e),(v_e)\bigl)$ is a nice $1$-skeleton cover of $\Omega$ by definition. The definition of the $t_e$, $u_e$, and $v_e$ makes also clear that $\overline{V_i}(\calC) \subseteq \Omega_i$ for all $i$.

By Lemma~\ref{lem:cover}, the sets $\conv(\overline{V_i}(\calC))$ form a cover of the $1$-skeleton of $\Omega$. Consider a two-dimensional face $f$ in $F\bigl(\conv(\overline{V_1}(\calC)),\conv(\overline{V_2}(\calC)),\conv(\overline{V_3}(\calC))\bigl)$. By Berge's theorem, the three $\Omega_i$ have a common point on $f$. Notice that the $t_e$, $u_e$, and $v_e$ have been picked in the relative interior of their $e$, and that $u_e$ and $v_e$ are always distinct. As noted above, this implies that 
\[
F\bigl(\conv(\overline{V_1}(\calC)),\conv(\overline{V_2}(\calC)),\conv(\overline{V_3}(\calC))\bigl) = F(V_1,V_2,V_3,\lambda)\, ,
\]
which concludes the proof.
\end{proof}

\subsection{Algorithms}
\label{sec:alg}

In the sequel, we establish Proposition~\ref{prop:omega_dim_1}, Theorem~\ref{thm:k_2}, and Theorem~\ref{thm:k_3} by introducing explicitly the linear programs mentioned in the statements. The proofs do not depend on the exact nature of each variable (integral or continuous), which is therefore not stated explicitly in the linear programs.

\subsubsection{When $\Omega$ is one-dimensional}
Without loss of generality, we assume $\Omega = [0,1]$. We introduce the following linear program.
\begin{equation}
    \label{eq:prob_part_reformulation_1d}\tag{T}
    \begin{array}{rll}
    \displaystyle\min_{\substack{\x,\y_1,\ldots,\y_k\\ \omega_1,\ldots,\omega_{k-1}}} & \boldc^\top \x + {\displaystyle\max_{i \in [k]}\boldd^{\top}\y_i} &\\[1ex]
    \text{s.t.}
    & A\x+B\y_{i}\leq \boldb(\omega_{i-1}) & \forall i\in [k] \\[1ex]
    & A\x+B\y_i\leq \boldb(\omega_i) & \forall i\in [k] \\[1ex]
    & \omega_{i-1} \leq \omega_i & \forall i\in [k] \, ,
    \end{array}
\end{equation}
where $\omega_0 = 0$ and $\omega_k = 1$.

\begin{proof}[Proof of Proposition~\ref{prop:omega_dim_1}]
Let $\x$, $\y_1,\ldots,\y_k$, $\omega_1,\ldots,\omega_{k-1}$ be a feasible solution of~\eqref{eq:prob_part_reformulation_1d}. The sets $\Omega_i \coloneqq [\omega_{i-1},\omega_i]$ form a cover of $\Omega$. The fact that $\y_i$ satisfies $A\x+B\y_i\leq \boldb(\omega)$ for all $\omega \in \Omega_i$ follows from the affine dependence of $\boldb$, and $\Omega_1,\ldots,\Omega_k$, $\x$, $\y_1,\ldots,\y_k$ form a feasible solution of~\eqref{prob:finite-adapt} with the same value of the objective function.

Conversely, let $\Omega_1,\ldots,\Omega_k$, $\x$, $\y_1,\ldots,\y_k$ be a feasible solution of~\eqref{prob:finite-adapt}. By Lemma~\ref{lem:convexity}, the $\Omega_i$'s can be assumed to be closed and convex, i.e., there are closed intervals of the form $[\omega_i^{\text{left}},\omega_i^{\text{right}}]$. Up to reindexing, we have $\omega_1^{\text{left}} \leq \omega_2^{\text{left}} \leq \cdots \leq \omega_{k-1}^{\text{left}}$. The intervals $[\omega_i^{\text{left}},\omega_{i+1}^{\text{left}}]$ are then included in $\Omega_i$ for all $i$, which implies that $\x$, $\y_1,\ldots,\y_k$, $\omega_1^{\text{left}},\ldots,\omega_{k-1}^{\text{left}}$ is a feasible solution of~\eqref{eq:prob_part_reformulation_1d}, with the same value of the objective function.
\end{proof}

\subsubsection{When $k=2$ and $k=3$}
\label{sec:alg_2_3}
We focus on the case $k=3$ because it is the most difficult case. The case $k=2$ will be discussed below.

For a cover of $V(\Omega)$ with three sets $V_1,V_2,V_3$ and a labeling $\lambda$ of the edges $e$ of $\Omega$ with non-empty subsets of $\{1,2,3\}$ satisfying conditions~\ref{single},~\ref{pair},~\ref{three}, we introduce the following linear program:
\begin{equation}
    \label{eq:prob_part_reformulation}
    \tag{$\text{P}_{V_1,V_2,V_3,\lambda}$}
    \begin{array}{rll}
    \displaystyle\min_{\substack{\x,\y_1,\y_2,\y_3\\ \boldt,\uu,\vv}} & \boldc^\top \x + {\displaystyle\max_{i =1,2,3}\boldd^{\top}\y_i} &\\[1ex]
    \text{s.t.}
    & A\x+B\y_i\leq \boldb(\omega) & \forall i \in \{1,2,3\} \;\;\; \forall \omega \in V_i \\[1ex]
    & A\x+B\y_i\leq \boldb(t_e) & \forall e \text{ s.t. $|\lambda(e)|=2$}\;\;\; \forall i \in \lambda(e) \\[1ex]
     & A\x+B\y_i\leq \boldb(u_e) & \forall e \text{ s.t. $|\lambda(e)|=3$}\;\;\; \forall i \text{ s.t. $e$ has no endpoint in $V_i$} \\[1ex]
     & A\x+B\y_i\leq \boldb(v_e) & \forall e \text{ s.t. $|\lambda(e)|=3$}\;\;\; \forall i \text{ s.t. $e$ has no endpoint in $V_i$} \\[1ex]
          & A\x+B\y_i\leq \boldb(u_e) & \forall e\text{ s.t. $|\lambda(e)|=3$}\;\;\; \forall i \text{ s.t. $e$ has its tail in $V_i$}  \\[1ex]
             & A\x+B\y_i\leq \boldb(v_e) & \forall e \text{ s.t. $|\lambda(e)|=3$}\;\;\; \forall i \text{ s.t. $e$ has its head in $V_i$}
          \\ [1ex]
    & A\x+B\y_i\leq \boldb(t_f) & \forall f \in F(V_1,V_2,V_3,\lambda) \;\;\; \forall i \in \{1,2,3\}  \\[0.7ex]
    & t_e \in e & \forall e \text{ s.t. $|\lambda(e)|=2$}\\[0.7ex]
    & u_e,v_e \in e & \forall e \text{ s.t. $|\lambda(e)|=3$} \\[0.7ex]
    & v_e \in [u_e,\text{head}(e)] & \forall e \text{ s.t. $|\lambda(e)|=3$} \\[0.7ex]
    & t_f \in f &\forall f \in F(V_1,V_2,V_3,\lambda) \, .
    \end{array}
\end{equation}
(The notation $F(V_1,V_2,V_3,\lambda)$ has been introduced in Section~\ref{subsubsec:three}.) Note that \eqref{eq:prob_part_reformulation} is indeed a linear program: the last four constraints, defining the range of $(t_e),(t_f),(u_e),(v_e)$, can easily be expressed linearly. There are at most $7^{|V(\Omega)|+|E(\Omega)|}$ such linear programs: there are $7$ possibilities for each vertex (which $V_i$'s does the vertex belong to?); there are at most $7$ possibilities for each edge (which non-empty subset of $\{1,2,3\}$ labels the edge?). The proof of Theorem~\ref{thm:k_3} establishes that there is always one of these linear programs whose optimal solution provides an optimal solution of the problem~\eqref{prob:finite-adapt}.

\begin{proof}[Proof of Theorem~\ref{thm:k_3}]
Let $\x$, $\y_1$, $\y_2$, $\y_3$, $\boldt$, $\uu$, $\vv$ be a feasible solution of the problem~\eqref{eq:prob_part_reformulation}, for some choice of the $V_i$'s and $\lambda$. Then the tuple $\calC =  \bigl((V_i),\lambda,(t_e),(u_e),(v_e)\bigl)$ is a nice $1$-skeleton cover. According to Lemma~\ref{lem:cover}, the sets 
$\conv(\overline{V_i}(\calC))$ form a cover of the $1$-skeleton of $\Omega$. Lemma~\ref{lem:cover-2} and inclusion~\eqref{eq:2f} together imply that the $\Omega_i$'s, defined by \[
\Omega_i\coloneqq \conv\bigl(C_i \cup \{t_f\colon f \in F(C_1,C_2,C_3)\}\bigl) \, ,
\]
form a cover of $\Omega$, where $C_i \coloneqq \conv(\overline{V_i}(\calC))$.  The fact that $\y_i$ satisfies $A\x+B\y_i\leq \boldb(\omega)$ for all $\omega \in \Omega_i$ follows from the affine dependence of $\boldb$, and $\Omega_1$, $\Omega_2$, $\Omega_3$, $\x$, $\y_1$, $\y_2$, $\y_3$ form a feasible solution of the problem~\eqref{prob:finite-adapt} with the same value of the objective function.

Conversely, let $\Omega_1$, $\Omega_2$, $\Omega_3$, $\x$, $\y_1$, $\y_2$, $\y_3$ be a feasible solution of the problem~\eqref{prob:finite-adapt}. By Lemma~\ref{lem:convexity}, the $\Omega_i$'s can be assumed to be closed and convex. Lemma~\ref{lem:3conv} shows that there exists a nice $1$-skeleton cover $\calC = \bigl((V_i),\lambda,(t_e),(u_e),(v_e)\bigl)$ such that $\overline{V_i}(\calC) \subseteq \Omega_i$. Set $t_f$ as any point common to the $\Omega_i$'s for each face $f$ in $F(V_1,V_2,V_3,\lambda)$, as ensured in addition by the lemma. Then $\x$, $\y_1$, $\y_2$, $\y_3$, $\boldt = (t_e),(t_f)$, $\uu = (u_e)$, $\vv = (v_e)$ form a feasible solution of the problem~\eqref{eq:prob_part_reformulation}, with the same value of the objective function.
\end{proof}

Theorem~\ref{thm:k_2} (case $k=2$) can be proved along the same lines, but in a much simpler way.  Especially, 
we simply define a nice $1$-skeleton cover as a cover of $V(\Omega)$ with two sets $V_1,V_2$, and a point $t_e$ in each edge $e$ whose endpoints do not belong to a same $V_i$. Lemma~\ref{lem:cover} remains the same. Lemma~\ref{lem:cover-2} is not needed anymore (Lemma~\ref{lem:cover_polytope} can be used directly). Lemma~\ref{lem:3conv} becomes: {\em Let $\Omega_1$, $\Omega_2$ be two closed convex sets covering $\Omega$. Then there exists a nice $1$-skeleton cover $\calC = \bigl((V_i),(t_e)\bigl)$ such that $\overline{V_i}(\calC) \subseteq \Omega_i$ for both $i$.} The linear program to consider is then
\begin{equation}\label{prog:k2}
    \tag{$\text{P}_{V_1,V_2}$}
    \begin{array}{rll}
    \displaystyle\min_{\substack{\x,\y_1,\y_2\\ \boldt}} & \boldc^\top \x + {\displaystyle\max_{i =1,2}\boldd^{\top}\y_i} &\\[1ex]
    \text{s.t.}
    & A\x+B\y_i\leq \boldb(\omega) & \forall i \in \{1,2\} \;\; \forall \omega \in V_i \\[1ex]
    & A\x+B\y_i\leq \boldb(t_e) & \forall i \in \{1,2\}\;\; \forall e \in E[V_1;V_2]  \\[1ex]
    & t_e \in e & \forall e \in E[V_1;V_2] \, ,
    \end{array}
\end{equation}
where $E[V_1;V_2]$ denotes the set of edges whose endpoints do not belong to a same $V_i$. There are at most $3^{|V(\Omega)|}$ such linear programs: there are $3$ possibilities for each vertex (which $V_i$'s does the vertex belong to?).

\subsubsection{An explicit mixed integer linear program for $k=2$}
\label{sec:explicit_milp}

In the case $k=2$, instead of solving a series of problems~\eqref{prog:k2} with various $V_1,V_2$, we can actually solve the single mixed integer linear program given below. This provides a concrete way of solving the case $k=2$ relying mostly on off-the-shelf solvers. The correctness of the following program is immediate from the previous section:
\label{sec:add_mat}
\begin{equation}
    \tag{$\text{P}_{k=2}$}
    \begin{array}{rll}
    \displaystyle\min_{\substack{\x,\y_1,\y_2\\ \boldt}} & \boldc^\top \x + {\displaystyle\max_{i =1,2}\boldd^{\top}\y_i} &\\[1ex]
    \text{s.t.}
    & A\x+B\y_i\leq \boldb(\omega) a_{i,\omega} + M (1-a_{i,\omega})\boldsymbol{1}& \forall i \in \{1,2\} \;\; \forall \omega \in V(\Omega) \\[1ex]
    & A\x+B\y_i\leq \boldb(t_e)+M(1-b_e)\boldsymbol{1} & \forall i \in \{1,2\} \;\; \forall e \in E(\Omega) \\[1ex]
    & a_{1,\omega} + a_{2,\omega} \geq 1 & \forall \omega \in V(\Omega) \\[1ex]
    & a_{1,\text{t}(e)} + a_{2,\text{h}(e)}  - a_{2,\text{t}(e)} - a_{1,\text{h}(e)}- 1 \leq b_e & \forall e \in E(\Omega) \\[1ex]
    & a_{2,\text{t}(e)} + a_{1,\text{h}(e)}  - a_{1,\text{t}(e)} - a_{2,\text{h}(e)}- 1 \leq b_e & \forall e \in E(\Omega) \\[1ex]
    & t_e \in e & \forall e \in E(\Omega)\\[1ex]
    &a_{i,\omega} \in \{0,1\} &\forall i \in \{1,2\} \;\; \forall \omega \in V(\Omega) \\[1ex]
    & b_e \in \{0,1\} & \forall e \in E(\Omega) \, .
    \end{array}
\end{equation}
Here, $\boldsymbol{1}$ is the all-one vector, t$(e)$ (resp.\ h$(e)$) is the tail (resp.\ head) of $e$, and $M$ is a big-M, i.e., a sufficiently large constant. It is not difficult to check that $M \coloneqq 2\max_{\omega \in \Omega} \|\boldb(\omega)\|_{\infty}$ is a valid choice. 

We finish by giving the meaning of the extra binary variables: $a_{i,\omega}$ takes the value $1$ if $\omega$ is in $V_i$, and $0$ otherwise; $b_e$ takes the value $1$ if $e$ is in $E[V_1;V_2]$, and $0$ otherwise.

A similar program could be written for $k=3$. The writing of such a program is quite easy, would follow the same logic as for $k=2$, but would be cumbersome.

\bibliographystyle{siamplain}
\bibliography{biblio}
\end{document}